\documentclass[12pt]{amsart}
\usepackage{amssymb} 
\newtheorem{thm}{Theorem}[section]

\newtheorem{lem}[thm]{Lemma}
\newtheorem{prop}[thm]{Proposition}

\newtheorem{qu}[thm]{Question}

\numberwithin{equation}{section}
\begin{document}
\title[Non-cyclic graph]{Non-cyclic graph associated with a group}
\author[A. Abdollahi \;\;  \&  \; A. Mohammadi
Hassanabadi]{{\bf A. Abdollahi$~^{1,2}$ \;\;\;  \&  \;\;\; A.
Mohammadi Hassanabadi$^{1,3}$}\\
~~\\
 $~^1$Department of Mathematics,\\ University of Isfahan,\\ Isfahan
81746-73441, Iran.\\
~~\\
$~^2$School of Mathematics, Institute for Research in Fundamental Sciences (IPM), P.O.Box: 19395-5746, Tehran, Iran.\\
~~\\
 $~^3$ Shaikhbahaee University}
 \thanks{e-mail Address' of Authors: {\tt
alireza\_abdollahi@yahoo.com},{\tt aamohaha@yahoo.com}}
\thanks{This research was in part supported by  Isfahan University Grant No. 851120  and its Center of Excellence for
Mathematics. The research of the first author was in part supported by a grant from IPM (No. 87200118)}
\begin{abstract}
We associate a graph $\mathcal{C}_G$  to a non locally cyclic
group $G$ (called the non-cyclic graph of $G$) as follows: take
$G\backslash Cyc(G)$ as vertex set, where $Cyc(G)=\{x\in G \;|\;
\langle x,y\rangle \; \text{is cyclic for all} \; y\in G\}$ is
called the cyclicizer of $G$, and join two vertices if they do not
generate a cyclic subgroup.
 For a simple graph $\Gamma$,
$w(\Gamma)$ denotes the clique number of $\Gamma$, which is the
maximum size (if it exists) of a complete subgraph of $\Gamma$.
In this paper we characterize groups whose non-cyclic graphs have
clique numbers at most $4$. We prove that a non-cyclic group $G$
is solvable whenever $w(\mathcal{C}_G)<31$ and the equality for a
non-solvable group $G$ holds if and only if $G/Cyc(G)\cong A_5$
or $S_5$.
\end{abstract}
\maketitle
\section{\bf Introduction and results}
 Let $G$ be a non locally cyclic group.
Following \cite{AM}, the non-cyclic graph $\mathcal{C}_G$ of $G$
is defined as follows: take  $G\backslash Cyc(G)$ as vertex set,
where $Cyc(G)=\{x\in G \;|\; \langle x,y\rangle \; \text{is cyclic
for all} \; y\in G\}$, and join two vertices if they do not
generate a cyclic subgroup. We call the complement of
$\mathcal{C}_G$, the cyclic graph of $G$, which has the same
vertex set as $\mathcal{C}_G$ and two distinct vertices are
adjacent whenever they generate a cyclic subgroup.
The cyclic graph of $G$ will be denoted by $\overline{\mathcal{C}_G}$.\\

We consider simple graphs which are
 undirected, with no loops or multiple edges. For any graph
$\Gamma$, we denote the sets of the vertices and the edges
  of $\Gamma$ by $V(\Gamma)$ and $E(\Gamma)$, respectively.
  The degree $d_\Gamma(v)$ of a vertex $v$ in $\Gamma$ is
the number of edges incident to $v$ and if the graph is
understood, then we denote $d_{\Gamma}(v)$ simply by $d(v)$. The
order of $\Gamma$ is defined $|V(\Gamma)|$. A graph $\Gamma$ is
{\it regular} if $d(v)=d(w)$ for any two vertices $v$ and $w$. A
subset $X$ of the vertices of $\Gamma$ is called a {\it clique}
if the induced subgraph on $X$ is a complete graph. The maximum
size of a clique in a graph $\Gamma$ is called the {\it clique
number} of $\Gamma$ and is denoted by $\omega(\Gamma)$.
 If there exists a path between two vertices $v$ and $w$  in $\Gamma$, then
$d_{\Gamma}(v,w)$ denotes the length of  the shortest path
between $v$ and $w$; otherwise $d_{\Gamma}(v,w)=\infty$. If the
graph is understood, then we denote $d_{\Gamma}(v,w)$ simply by
$d(v,w)$. The largest distance between all pairs of the vertices
of $\Gamma$ is called the {\it diameter} of $\Gamma$, and is
denoted by $\mathrm{diam}(\Gamma)$. A graph $\Gamma$ is {\it
connected} if there is a path between each pair of the vertices
of $\Gamma$. So  disconnected graphs have infinite diameter.
 For a graph $\Gamma$ and a subset $S$ of
the vertex set $V(\Gamma)$, denote by $N_\Gamma[S]$ the set of
vertices in $\Gamma$ which are in $S$ or adjacent to a vertex in
$S$. If $N_\Gamma[S] = V(\Gamma)$, then $S$ is said to be a
dominating set (of vertices in $\Gamma$).  The {\it domination
number} of a graph $\Gamma$, denoted by $\gamma(\Gamma)$, is the
minimum size of a dominating set of the vertices in $\Gamma$.
 A {\it planar} graph is a graph that can be embedded in the plane so
that no two edges intersect geometrically except at a vertex at
which both are incident. We denote the symmetric group on $n$
letters and the alternating  group of degree $n$ by $S_n$ and
$A_n$, respectively. Also $Q_8$ and $D_{2n}$ are used for the
quaternion group with 8 elements and dihedral group of order $2n$
($n>2$),
respectively.\\

The present work is a continuation of that of \cite{AM}.
 In section 2, we study the diameter and domination number of
the cyclic and non-cyclic graphs. In section 3, we characterize
all groups whose non-cyclic graphs have  clique numbers $\leq 4$.
In section 4, we classify all groups whose non-cyclic graphs are
planar or Hamiltonian. Finally in section 5, we give a sufficient
condition for solvability, by proving that a group $G$ is
solvable whenever $\omega(\mathcal{C}_G)<31$. We also prove the
bound $31$ cannot be improved and indeed  the equality for a
non-solvable group $G$ holds if and only if $G/Cyc(G)\cong A_5$
or $S_5$.
\section{\bf On the diameter and domination numbers of  the non-cyclic graph and its complement}
We first observe that to study the diameter of the non-cyclic
graph, we may  factor out the cyclicizer. Recall that if there
exists a path between two vertices $v$ and $w$  in a graph
$\Gamma$, then $d(v,w)$ denotes the length of  the shortest path
between $v$ and $w$; otherwise $d(v,w)=\infty$. The largest
distance between all pairs of the vertices of $\Gamma$ is called
the {\it diameter} of $\Gamma$, and is denoted by
$\mathrm{diam}(\Gamma)$. Thus if $\Gamma$ is disconnected then
$\mathrm{diam}(\Gamma)=\infty$.
\begin{lem}\label{d=3}
Let $G$ be a non locally cyclic group. Then
$\overline{\mathcal{C}_{G}}$ is connected if and only if
$\overline{\mathcal{C}_{\frac{G}{Cyc(G)}}}$ is connected,
$\mathrm{diam}(\mathcal{C}_{G})=\mathrm{diam}\big(\mathcal{C}_{\frac{G}{Cyc(G)}}\big)$
and
$\mathrm{diam}(\overline{\mathcal{C}_{G}})=\mathrm{diam}\big(\overline{\mathcal{C}_{\frac{G}{Cyc(G)}}}\big)$.
Moreover, corresponding connected components of
$\overline{\mathcal{C}_{G}}$ and
$\overline{\mathcal{C}_{\frac{G}{Cyc(G)}}}$ have the same diameter
when the component in $\overline{\mathcal{C}_{\frac{G}{Cyc(G)}}}$
is not an isolated vertex.
\end{lem}
\begin{proof}
It is enough to prove that $x-y$ is an edge in $\mathcal{C}_G$ if
and only if $\bar{x}-\bar{y}$ is an edge in
$\mathcal{C}_{\frac{G}{Cyc(G)}}$, where $\bar{}$ is the natural
epimorphism from $G$ to $\frac{G}{Cyc(G)}$. If $x-y$ is an edge in
$\mathcal{C}_G$, then $\langle x,y\rangle$ is not cyclic. We have
to prove that $\langle \bar{x},\bar{y}\rangle$ is not cyclic.
Suppose, for a contradiction, that $\langle
\bar{x},\bar{y}\rangle$ is cyclic. Then $x=g^ic_1$ and $y=g^jc_2$
for some $g\in G$, $c_1,c_2\in Cyc(G)$ and integers $i,j$. Thus
$$\langle x,y\rangle=\langle g^ic_1,g^jc_2\rangle\leq \langle g,c_1,c_2\rangle.$$
Now since $\langle c_1,c_2 \rangle=\langle c \rangle$, for some
$c\in Cyc(G)$,  it follows that $\langle g,c_1,c_2\rangle$ is
cyclic.
Therefore $\langle x,y \rangle$ is cyclic, a contradiction.\\
Now if $\bar{x}-\bar{y}$ is an edge in
$\mathcal{C}_{\frac{G}{Cyc(G)}}$, then $\langle
\bar{x},\bar{y}\rangle$ is not cyclic; and since  $\langle
\bar{x},\bar{y}\rangle$ is a homomorphic image of $\langle x,y
\rangle$, $\langle x,y\rangle$ is not cyclic. This completes the
proof.
\end{proof}
\begin{lem}\label{d(x,y)=3}
Let $G$ be a finite non-cyclic group and let $x,y\in G\backslash
Cyc(G)$. Then $d_{\mathcal{C}_G}(x,y)=3$ if and only if
$G=Cyc_G(x) \cup Cyc_G(y)$. Moreover, $\langle x,y\rangle$ is
cyclic and  for all $t\in Cyc_G(x)\backslash Cyc_G(y)$ and for
all $s\in Cyc_G(y) \backslash Cyc_(x)$, $\langle t,s\rangle$ is
not cyclic.
\end{lem}
\begin{proof}
The proof is contained in that of Proposition 3.2 of \cite{AM}.
\end{proof}
Recall that if  $G$ is a non locally cyclic group, then  two
distinct vertices are adjacent in the cyclic graph
$\overline{\mathcal{C}_{G}}$   if and only if they generate a
cyclic group.
\begin{lem}\label{com1}
Let $G$ be a non locally cyclic group. Then
$\mathrm{diam}(\overline{\mathcal{C}_G})\not=1$. In other words,
$\overline{\mathcal{C}_G}$ cannot  be isomorphic to a complete
graph.
\end{lem}
\begin{proof}
If $\mathrm{diam}(\overline{\mathcal{C}_G})=1$, then every two
elements of $G$ generates a cyclic group. This is equivalent to
$G$ being  locally cyclic, a contradiction.
\end{proof}
\begin{prop}
\begin{enumerate}
\item If $G$ is a non locally cyclic  group such that either
$Cyc(G)\not=1$ or $\overline{\mathcal{C}_G}$ is connected and
$Cyc(G)=1$, then $G$ is either a torsion group or  a torsion free
group.
\item If $A$ is an abelian torsion-free non locally cyclic group, then $\overline{\mathcal{C}_A}$ is disconnected.
\item
There are torsion-free simple groups $H$ such that
$\mathrm{diam}(\overline{\mathcal{C}_H})=\mathrm{diam}(\mathcal{C}_H)=2$.
\end{enumerate}
\end{prop}
\begin{proof}
(1) \;  If $Cyc(G)\not=1$, then the proof follows  from Lemma 2.3
of \cite{AM}. Thus assume that $\overline{\mathcal{C}_G}$ is
connected and $Cyc(G)=1$. If there were elements of infinite
order and non-trivial elements of finite order, then connectivity
would guarantee some pair of these would be adjacent in
$\overline{\mathcal{C}_G}$, which is not possible.
 This proves (1).\\
(2) \; Suppose, for a contradiction, that
$\overline{\mathcal{C}_A}$ is connected. Note that any two
adjacent vertices $a, b$ satisfy $\langle a\rangle \cap \langle b
\rangle\not=1$. Since  $A$ is torsion-free and
$\overline{\mathcal{C}_A}$ is connected, it follows that $\langle
a\rangle \cap \langle b \rangle\not=1$ for any two non-trivial
elements $a,b$ of $A$. Fix a non-trivial element $a\in A$, then
it is easy to see that the map  $f$ defined from $A$ to the
additive group $\mathbb{Q}$ of rational numbers  by $f(x)=m/n$,
where $x^n=a^m$, is a group monomorphism, where $m$ and $n$ are
integers such that $1\not=x^m=a^n\in \langle x\rangle \cap
\langle a \rangle$. Therefore $A$ is isomorphic to a subgroup of
the additive group of rational numbers and so it is locally
cyclic, which is impossible. \\
 (3) \; In \cite{Ob} a torsion-free simple group $H$ is
constructed such that the intersection of any two of its
non-trivial subgroups is non-trivial. Therefore $Cyc(H)=1$ and
for any two non-trivial elements $a$ and $b$ of $H$, $\langle
a\rangle \cap \langle b\rangle \neq 1$. This  together with Lemma
\ref{com1} implies that
$\mathrm{diam}(\overline{\mathcal{C}_H})=2$. On the other hand,
by Proposition 3.7 of \cite{AM},
$\mathrm{diam}(\mathcal{C}_H)=2$. This completes the proof.
\end{proof}
\begin{lem}
Let $G$ be a finite non-cyclic group of prime power order. Then
$\overline{\mathcal{C}_G}$ is disconnected.
\end{lem}
\begin{proof}
Suppose, for a contradiction, that $\overline{\mathcal{C}_G}$ is
connected. By Lemma \ref{d=3}, we may assume that $Cyc(G)=1$.  Now
we prove that $G$ has only one subgroup of prime order. Suppose
that there are two elements $a$ and $b$  of prime order. Since
$\overline{\mathcal{C}_G}$ is connected, there exists a sequence
$x_1,\dots,x_n$ of elements of $G\backslash Cyc(G)$ such that
$$\langle a,x_1\rangle,\langle
x_1,x_2\rangle,\dots,\langle x_{n-1},x_n\rangle,\langle
x_n,b\rangle\eqno{(1)}$$ are all cyclic. Since a cyclic group of
$p$-power order has only one subgroup of order $p$ and  both
$\langle a\rangle$ and $\langle b\rangle$ are subgroups of order
$p$ of the cyclic groups $(1)$, we have that $\langle
a\rangle=\langle b \rangle$. Now let $A$ be the only subgroup of
prime order in $G$, which can be generated by $x$ and let $y$ be
any non-trivial element of $G$. Then  $\langle x,y\rangle=\langle
y\rangle$. This shows that $x\in Cyc(G)$, which is impossible.
This completes the proof.
\end{proof}
\begin{lem}\label{d-comp=3}
Let $G$ be a  non locally cyclic group. If
$\mathrm{diam}(\mathcal{C}_{G})=3$ then
$\overline{\mathcal{C}_G}$ is connected and
$\mathrm{diam}\big(\overline{\mathcal{C}_{G}}\big)\in\{2,3\}$.
\end{lem}
\begin{proof}
 Let $x,y\in G\backslash Cyc(G)$ such that $d_{\mathcal{C}_G}(x,y)=3$. By Lemma \ref{d(x,y)=3},
 we have $G=C_x \cup C_y$, where $C_x=Cyc_G(x)$ and $C_y=Cyc_G(y)$.
Now let $a$ and $b$ be two distinct elements of $G\backslash
Cyc(G)$. If $K=\langle a,b \rangle$ is cyclic, then
$d_{\overline{\mathcal{C}_G}}(a,b)=1$. Suppose that $K$ is
non-cyclic. We may assume without loss of generality that $a\in
C_x\backslash C_y$ and $b\in C_y\backslash C_x$, (otherwise
$d_{\overline{\mathcal{C}_G}}(a,b)=2$ as either $a-x-b$ or
$a-y-b$ is a path of length two in $\overline{\mathcal{C}_G}$).
In this case, $a-x-y-b$ is a path of length 3 in
$\overline{\mathcal{C}_G}$. Now Lemma \ref{com1} completes the
proof.
\end{proof}
We have checked by {\sf GAP} \cite{GAP}, that for each finite
non-cyclic group $G$ of order at most 100, the following holds
$$\mathrm{diam}(\mathcal{C}_{G})=3 \Longleftrightarrow \mathrm{diam}\big(\overline{\mathcal{C}_{G}}\big)=3.$$
We were unable to prove the equality
$\mathrm{diam}\big(\overline{\mathcal{C}_{G}}\big)=3$ in Lemma
\ref{d-comp=3} for all non locally cyclic groups $G$. So we may
pose the following question:
\begin{qu}
In Lemma $\ref{d-comp=3}$, for which non locally cyclic group
$G$  does the equality
$\mathrm{diam}\big(\overline{\mathcal{C}_{G}}\big)=3$  holds?
\end{qu}
The following is an example of a finite non-cyclic group $G$ with
$\mathrm{diam}(\mathcal{C}_G)=2$ and
$\mathrm{diam}(\overline{\mathcal{C}_G})=4$. Let $G = C_2 \times
F$ be the direct product of a cyclic group of order 2 generated
by $z$ say, with a Frobenius group $F$ of order $6 \cdot 7$
(which is not the dihedral group $D_{42}$). A Sylow 3-subgroup
(there are seven of these) is cyclic of order 3, and if $P$ and
$Q$ are two distinct ones, then $C_G(P) \cap C_G(Q) = \langle
z\rangle$. In particular, if $x$ and $y$ are two non-central
elements of $G$, then $x$ fails to centralize at least 6 Sylow
3-subgroups, and then $x$ and $y$ together fail to centralize at
least 5 of these. Thus, the distance $d_{\mathcal{C}_G}(x, y)$ in
the non-cyclic graph is at most 2. On the other hand, if exactly
one of these elements, say $x$, is central (so $x=z$), then
choose a Sylow 3-subgroup $P$ which does not centralize $y$, and
then choose an element $g$ of order 6 in the centralizer
$C_G(P)$. Then we have the path $x-g-y$ of length 2 in the non-cyclic graph of $G$. This establishes
$\mathrm{diam}(\mathcal{C}_G)=2$.\\
In the cyclic graph $\overline{\mathcal{C}_G}$, every element has
distance at most 2 from the central element $z$. Certainly,
elements of odd order are directly adjacent to $z$, elements of
even order $\not= 2$ are connected to elements of odd order, so
have distance $\leq 2$ from $z$, while a non-central involution
centralizes some (unique) Sylow 3-subgroup of $G$ so that it too
has distance $\leq 2$ from $z$. Hence
$\mathrm{diam}(\overline{\mathcal{C}_G})\leq 4$. It remains to
find two elements $u$ and $v$ whose distance is exactly 4 in the
cyclic graph $\overline{\mathcal{C}_G}$.\\ Choose $u$ and $v$ to
be non-central involutions centralizing two distinct Sylow
3-subgroups, say $P=\langle g\rangle$ and $Q=\langle h\rangle$,
respectively. Certainly $u-g-z-h-v$ is a path of length 4 in the
cyclic graph $\overline{\mathcal{C}_G}$, and we argue that there
is no shorter path from $u$ to $v$. Clearly, $u\not=v$ and $u, v$
are not adjacent in the cyclic graph. Moreover, $C_G(u)\cap
C_G(v)=\langle z\rangle$ shows that there is no path of length 2
from $u$ to $v$ (as neither $\langle u, z\rangle$ nor $\langle v,
z\rangle$ is cyclic). Furthermore, if $x$ is any element adjacent
to $u$, then $\langle x\rangle$ is either $P$ or $P\langle
u\rangle$. Therefore, in any path from $u$ to any other element,
say $u-x-\cdots$ we may replace $x$ by an appropriate generator
of $P$. If this path ends at $v$, then the ending $\cdots y-v$
may be adjusted similarly so that $y$ is a generator of $Q$. As
$\langle x, y\rangle$ is not cyclic, the total length of the path
is $\geq
4$.\\

Two other examples are  {\tt SmallGroup(48,11)} and {\tt
SmallGroup(48,12)} in {\sf GAP} \cite{GAP}.
 It is also checked that for all non-cyclic
groups $G$ of order at most 100, either
$\overline{\mathcal{C}_G}$ is disconnected or
$\mathrm{diam}(\overline{\mathcal{C}_G})\in\{3,4\}$.\\
\begin{lem}
Let $G$ be a non locally cyclic group. Then
\begin{enumerate}
\item $\gamma(\overline{\mathcal{C}_G})\geq 2$. The equality holds
if and only if $\mathrm{diam}(\mathcal{C}_G)=3$.
\item $\gamma(\mathcal{C}_G)=1$ if and only if $Cyc(G)=1$ and
there is an element $x$ of order $2$ such that $Cyc_G(x)=\langle
x \rangle$.
\item  If  either
$G=E*H$ is the free product of a non-trivial elementary abelian
$2$-group $E$
 with an arbitrary group $H$ such that either $|E|>2$ or $|H|>1$;
or $G$ has an abelian  $2'$-subgroup $A$
 and an element $x$ of order $2$ such that $G=A\langle x\rangle$ and  $a^x=a^{-1}$ for all $a\in
 A$,  then
 $\gamma(\mathcal{C}_G)=1$.
\end{enumerate}
\end{lem}
\begin{proof}
(1) \; Suppose, for a contradiction, that there is a dominating
singleton set $\{x\}$ for $\overline{\mathcal{C}_G}$. Then for
all $a\in G\backslash Cyc(G)$ we have $a=x$ or $\langle a,x
\rangle$ is cyclic. It follows that $\langle a,x\rangle$ is
cyclic for all $a\in G$ and so $x\in Cyc(G)$, a contradiction. \\
Now suppose that $\gamma(\overline{\mathcal{C}_G})=2$. Then there
exist two distinct vertices $x$ and $y$ of
$\overline{\mathcal{C}_G}$ such that for every vertex
$a\not\in\{x,y\}$, either $\langle a,x\rangle$  or $\langle
a,y\rangle$ is cyclic. This implies that $G=Cyc_G(x) \cup
Cyc_G(y)$. Now Lemma \ref{d(x,y)=3} and \cite[Proposition 3.2]{AM}
complete the proof.\\
 (2) \; Suppose that $\mathcal{C}_G$ has a
dominating singleton set $\{x\}$. Since $\langle x,x^{-1}\rangle$
is trivially  cyclic,
$x=x^{-1}$ and so $x^2=1$.\\
If $t\in Cyc(G)$ and $t\not=1$, then $\langle tx,x\rangle$ is
cyclic. It follows that $tx=x$ and so $t=1$. Thus $Cyc(G)=1$.\\
If $c\in Cyc_G(x)$ and $c\not=x$, then $\langle c,x\rangle$ is
cyclic. This implies that $c\in Cyc(G)=1$ and so
$Cyc_G(x)=\langle x \rangle$.\\
For the converse, it is enough to note that for all $a\in
G\backslash \{1,x\}$, $\langle a,x\rangle$ is not cyclic. This
shows that $\{x\}$ is a dominating set for $\mathcal{C}_G$ and so
$\gamma(\mathcal{C}_G)=1$. \\
(3) \; Suppose that $G=E*H$ is the free product of an elementary
abelian $2$-group $E$  with an arbitrary group $H$ such that either
$|E|>2$ or $|H|>1$. Let $x$ be an arbitrary non-trivial element
of $E$. Then the centralizer $C_G(x)$ of $x$ in $G$ is equal to
$E$ and since $E$ is elementary abelian, we have that
$Cyc_E(x)=\langle x\rangle$. It follows that $Cyc_G(x)=\langle
x\rangle$. Now by part (2) it is enough to show that $Cyc(G)=1$. If $Z(G)=1$, then obviously $Cyc(G)=1$. If $Z(G)\not=1$, then
 $|H|=1$ as $|E|>2$. Thus $G=E$ and since $|E|>2$, there are two non-trivial distinct elements $a$ and $b$ in $E$. Since every non-trivial element of $E$ has order $2$,  $Cyc_E(g)=\langle g\rangle$ for all non-trivial elements $g\in E$. Thus  $$Cyc(G)=Cyc(E)\leq \langle a\rangle \cap \langle b\rangle=1,$$ as required. \\
It is straightforward to see that if $G$ is of second type, then
the singleton $\{x\}$ is a dominating set for $\mathcal{C}_G$.
\end{proof}
\section{\bf Finite groups whose non-cyclic graphs have small clique numbers}
In this section we characterize groups whose non-cyclic graphs
have  clique numbers at most $4$. If $\{x_1,x_2,\dots,x_n\}$ is a
maximal clique for the finite group $G$ then each $x_i$ is
contained in a (unique) maximal cyclic subgroup. Replacing each
$x_i$ by a generator of this maximal cyclic subgroup does no harm,
and the resulting collection of cyclic subgroups $\langle
x_1\rangle,\dots, \langle x_n\rangle$ is a complete list of all
the maximal cyclic subgroups, by Theorem 4.7 of \cite{AM}.
\begin{lem}\label{w>2}
Let $G$ be a non locally cyclic group. Then
$\omega(\mathcal{C}_G)\geq 3$.
\end{lem}
\begin{proof}
Since $G$ is not locally cyclic, there exists two elements $x$
and $y$ such that $\langle x,y\rangle$ is not cyclic. Thus
$\{x,y,xy\}$ is a clique in $\mathcal{C}_G$. This completes the
proof.
\end{proof}
\begin{lem}\label{w-cyc}
Let $G$ be a non locally cyclic group whose non-cyclic graph has
no  infinite clique. Then $\omega(\Gamma_G)$ is finite and
$\omega(\mathcal{C}_G)=\omega(\mathcal{C}_{\frac{G}{Cyc(G)}})$.
\end{lem}
\begin{proof}
It follows from Theorem 4.2 and Lemma 2.3-(2) of \cite{AM}.
\end{proof}
Thus by Lemma \ref{w-cyc} and Lemma 2.3-(2) of \cite{AM}, to
characterize groups $G$ with finite fixed $\omega(\mathcal{C}_G)$,
it is enough to characterize finite ones with trivial
cyclicizers.\\ We use the following result in the proof of
Theorems \ref{2.7} and \ref{non-sol-31}
\begin{lem}\label{G>G/N}
Let $G$ be a non locally cyclic group such that
$\omega(\mathcal{C}_G)$ is finite. If $N$ is a normal subgroup of
$G$ such that  $G/N$ is not locally cyclic, then
$\omega(\mathcal{C}_{\frac{G}{N}})\leq \omega(\mathcal{C}_{G})$,
with equality if and only if $N\leq Cyc(G)$.
\end{lem}
\begin{proof}
Let $\omega(\mathcal{C}_G)=n$ and $\overline{G}=G/N$.  If
$L/N=Cyc(\overline{G})$, then $Cyc(G)N\leq L$ and so by Lemma
\ref{w-cyc}
$$\omega\big(\mathcal{C}_{\frac{\overline{G}}{Cyc(\overline{G})}}\big)=\omega\big(\mathcal{C}_{\frac{G}{L}}\big)\leq
\omega\big(\mathcal{C}_{\frac{G}{Cyc(G)N}}\big)\leq
\omega\big(\mathcal{C}_{\overline{G}}\big)=\omega\big(\mathcal{C}_{\frac{\overline{G}}{Cyc(\overline{G})}}\big).$$
Since by Theorem 4.2 of \cite{AM}, $G/Cyc(G)$ is finite, without
loss of generality, we may assume that $G$ is  finite.\\
 Clearly $\omega(\mathcal{C}_{G/N})\leq \omega(\mathcal{C}_G)$. Now suppose that $\omega(\mathcal{C}_{G/N})=\omega(\mathcal{C}_G)$. Then there exist
 elements  $y_1,\dots,y_{n}\in G$  such that
$M=\{y_iN\;|\; i=1,\dots,n\}$ is a clique of
$\mathcal{C}_{\overline{G}}$. Choose now  a maximal cyclic
subgroup $C_i$ of $G$ containing $y_i$ ($C_i$ is in fact uniquely
determined by $y_i$). There is no harm in replacing each $y_i$ by
a generator $x_i$ of $C_i$. Now it follows from Theorem 4.7 of
\cite{AM} that  $C_1,\dots,C_n$ are all the maximal cyclic
subgroups of $G$. Consider an arbitrary element $a\in N$. Then
$$\{x_1,x_2\dots,x_{n},ax_1\}$$ is not a
clique of $\mathcal{C}_{G}$. Since $M$ is a clique for
$\mathcal{C}_{\overline{G}}$, it follows that
$$\langle x_1,ax_1\rangle=\langle a,x_1\rangle \;\text{is cyclic for all}\; a\in N. $$
This says that  $a\in\langle a,x_1\rangle=C_1=\langle x_1\rangle$
for all $a\in N$. But $x_1$ may be replaced by any of the $x_i$,
and we conclude  that  $N\leq \displaystyle\bigcap_{i=1}^n C_i=
Cyc(G)$. This completes the proof.
\end{proof}
Throughout for a prime number $p$ we denote by $\nu_p(G)$ the
number of subgroups of order $p$ in a group $G$. It is well-known
that $\nu_p(G)\equiv 1 \;\text{mod}\; p$ for a finite group $G$,
whenever $p$ divides $|G|$.
\begin{lem}\label{2.2}
Let $G$ be a finite group  with trivial cyclicizer. Then for
any prime divisor $p$ of $|G|$,
$\nu_p(G)\leq\omega(\mathcal{C}_G)$. If $p_1,\dots,p_k$ are
distinct prime numbers such that $G$ has no element of order
$p_ip_j$ for all distinct $i,j$, then $\sum_{i=1}^k
\nu_{p_i}(G)\leq \omega(\mathcal{C}_G)$.
\end{lem}
\begin{proof}
Let $C_1,\dots,C_{\nu_p(G)}$ be all the subgroups of order $p$ of
$G$. If $c_i$ is a generator of $C_i$, then
$\{c_1,\dots,c_{\nu_p(G)}\}$ is a clique in $\mathcal{C}_G$. Thus
$\nu_p(G)\leq \omega(\mathcal{C}_G)$, as required. To prove the
second part, for every $i\in\{1,\dots,k\}$ and every subgroup of
order $p_i$, take a generator of the subgroup, then the set
consisting of these generators is a clique in $\mathcal{C}_G$.
This completes the proof.
\end{proof}
\begin{lem}\label{2.3}
Let $G$ be a finite group $G$ with trivial cyclicizer. Let $p$ be
a prime number such that $p^{k-1}<\omega(\mathcal{C}_G)\leq p^k$,
for some $k\in\mathbb{N}$.
\begin{enumerate}
\item For every $p$-element $x$ of $G$, we have $x^{p^{k-1}}\in
Z(G)$.
 \item If $k=1$, then $G$ has no non-trivial $p$-element.
\item No Sylow $p$-subgroup of $G$ is  cyclic of order
greater than $p^{k-1}$.
\end{enumerate}
\end{lem}
\begin{proof}
(1) \; Let $n=\omega(\mathcal{C}_G)$ and suppose that
$x^{p^{k-1}}\not=1$. The goal is to show that every $y\in G$
centralizes $x^{p^{k-1}}$. This is obvious, if $y\in\langle
x\rangle$  so  assume $y\in G -\langle x\rangle$. Then with
$X=\{x\} \cup \langle x\rangle y$, it is clear that $|X|\geq
n+1$. Next, as some two element subset of $X$ generates a cyclic
group, there are only two cases to consider. If one of these two
elements is $x$, then the cyclic subgroup in question is $\langle
x, x^iy\rangle=\langle x, y\rangle$, so clearly $y\in
C_G(x)\subseteq C_G(x^{p^{k-1}})$. If on the other hand the
elements are $x^iy$ and $x^jy$, then since $$\langle x^iy,
x^jy\rangle=\langle x^iy(x^jy)^{-1}, x^jy\rangle=\langle x^{i-j},
x^jy\rangle,$$ we conclude that $x^jy \in C_G(x^{i-j})\subseteq
C_G(x^{p^{k-1}})$, so clearly $y$ belongs to this last set as
well.\\
(2) \; In the proof of part (1), put $k=1$ . Since $i-j<p$,
$\gcd(i-j,p)=1$ and so $\left<x,y\right>$ is cyclic for all $y\in
G$. Thus $x\in
Cyc(G)$=1. This completes the proof of part (2).\\
(3) \; Suppose, for a contradiction, that $G$ has a cyclic Sylow
$p$-subgroup of order greater than $p^{k-1}$. Then
 by part (1),   $x^{p^{k-1}}\in Z(G)$ for every
$p$-element $x$ of $G$. Since  Sylow $p$-subgroups of $G$ are
cyclic, it follows that $\langle x^{p^{k-1}},y\rangle$ is cyclic
for all $y\in G$. This implies that  $x^{p^{k-1}}\in Cyc(G)=1$
for all $p$-elements $x\in G$, which gives a contradiction.
\end{proof}
For a group $G$, we denote the non-commuting graph of $G$ by
$\mathcal{A}_G$. This is the graph whose vertex set is $G\setminus
Z(G)$ and two vertices $x$ and $y$ are adjacent if $xy\neq yx$.
This graph was studied in \cite{AAM} and \cite{moghadam}.
\begin{lem}  \label{2.4} Let $G$ be an abelian group. Then
$\omega(\mathcal{C}_G)=3$ if and only if $G\cong
\mathbb{Z}_2\oplus \mathbb{Z}_2 \oplus T$, where $T\cong Cyc(G)$
is a locally cyclic torsion group in which all elements have odd
order.
\end{lem}
\begin{proof} Suppose that $\omega(\mathcal{C}_G)=3$ and  $\overline{G}=G/Cyc(G)$. Since  $Cyc(\overline{G})=1$, then by Lemma \ref{2.3},
$\overline{G}$ is a $2$-group. Thus $\overline{G}\cong
\mathbb{Z}_{2^{\alpha_1}}\oplus\cdots \oplus
\mathbb{Z}_{2^{\alpha_k}}$, and as $\omega(\mathcal{C}_G)=3$, we
have $k=2$ and $\alpha_1=\alpha_2=1$. Therefore $\overline{G}\cong
\mathbb{Z}_2\oplus\mathbb{Z}_2$. Now it follows from parts (4)
and (5) of  Lemma 2.3 of \cite{AM} that $G$ is  torsion. If
$Cyc(G)$ contains an element of order $2$, then $G$ contains a
subgroup isomorphic to either $\mathbb{Z}_2\oplus \mathbb{Z}_2
\oplus \mathbb{Z}_2$ or $\mathbb{Z}_2\oplus \mathbb{Z}_4$, which
is not possible. Thus all elements of $Cyc(G)$ have odd order and
so $Cyc(G)$ is the $2'$-primary component of $G$ and so  $G\cong
\mathbb{Z}_2\oplus \mathbb{Z}_2 \oplus Cyc(G)$.\\ The converse is
clear.
\end{proof}
\begin{thm}\label{2.5}  Let $G$ be a non locally cyclic group. Then
$\omega(\mathcal{C}_G)=3$ if and only if $G/Cyc(G)\cong
\mathbb{Z}_2\oplus \mathbb{Z}_2$.
\end{thm}
 \begin{proof} Suppose that
$\omega(\mathcal{C}_G)=3$. By Lemma \ref{w-cyc}, we may assume
that $Cyc(G)=1$. Also it follows from Lemma \ref{2.3}, that $G$
is a $2$-group, and by Lemma \ref{2.4}, we may assume that $G$ is
a non-abelian group. Now since $\omega(\mathcal{C}_G)\geq
\omega(\mathcal{A}_G)$, we have $\omega(\mathcal{A}_G)=3$. Thus by a
well-known result (see e.g., \cite[Lemma 2.4-(3)]{AJM} and
\cite[Theorem 2]{BS}), we have $G/Z(G)\cong \mathbb{Z}_2\oplus
\mathbb{Z}_2$. Since $G$ is nilpotent,  $Z(G)\neq Cyc(G)=1$. Then
there exists an element $a\in G$ such that $H=\langle a,
Z(G)\rangle$ is not cyclic. Now since $H$ is abelian and since
$\omega(\mathcal{C}_H)\leq w(\mathcal{C}_G)=3$, it follows from
Lemma \ref{2.4} that  $H\cong \mathbb{Z}_2\oplus \mathbb{Z}_2$.
Thus $Z(G)\cong \mathbb{Z}_2$ or $\mathbb{Z}_2\oplus
\mathbb{Z}_2$. If $Z(G)\cong \mathbb{Z}_2\oplus \mathbb{Z}_2$,
then taking $a'\in G\setminus Z(G)$, again $\langle
a',Z(G)\rangle$ is abelian, and by a similar argument, it is
isomorphic to $\mathbb{Z}_2\oplus \mathbb{Z}_2$, which is a
contradiction. Thus $Z(G)\cong \mathbb{Z}_2$, and so $|G|=8$.
Therefore $G\cong D_8$ or $Q_8$, but $Cyc(Q_8)\neq 1$. Thus
$G\cong D_8$, another contradiction, as
$\omega(\mathcal{C}_{D_8})=5$; since if $D_8=\langle a,b\mid
a^4=b^2=1, bab=a^{-1}\rangle$, then $\{a, b, ab, a^2b,a^3b\}$ is a
clique in $\mathcal{C}_{D_8}$. \\ The converse is clear.
\end{proof}
\begin{lem}\label{ele}
Let $G$ be a  group of size $p^n$ and exponent $p$, where $p$ is a
prime number and $n>1$ is an integer. Then
$\omega(\mathcal{C}_G)=\frac{p^n-1}{p-1}$.
\end{lem}
\begin{proof}
For $x,y\in G$, $\langle x,y\rangle$ is not cyclic if and only if
$\langle x\rangle\not=\langle y \rangle$. Thus
$\omega(\mathcal{C}_G)$ is equal to the number of subgroups of $G$
of order $p$. This completes the proof.
\end{proof}
\begin{lem}\label{cl-cl}
Let $G$ and $H$ be two finite non-cyclic groups such that
$\gcd(|G|,|H|)=1$. If $\omega(\mathcal{C}_G)=n$ and
$\omega(\mathcal{C}_H)=m$ are finite, then
$\omega(\mathcal{C}_{G\times H})\geq nm$.
\end{lem}
\begin{proof}
Let $\{g_1,\dots,g_n\}$ and $\{h_1,\dots,h_m\}$ be two clique sets
in $G$ and $H$, respectively. Now it is easy to see that the set
$$\{(g_i,h_j) \;|\; i\in\{1,\dots,n\}, \; j\in\{1,\dots,m\}\}$$
is a clique  in $\mathcal{C}_{G\times H}$. This completes the
proof.
\end{proof}
\begin{lem} \label{2.6} Let $G$ be an abelian group such that
 $\omega(\mathcal{C}_G)=4$. Then $G/Cyc(G)\cong \mathbb{Z}_3\oplus
 \mathbb{Z}_3$.
\end{lem}
 \begin{proof}  By Lemmas \ref{w-cyc} and \ref{2.3}
and Lemma 2.3-(2) of \cite{AM}, $H=G/Cyc(G)$ is an abelian $\{2,
3\}$-group with Sylow $p$-subgroups of exponent $p$ and
$\omega(\mathcal{C}_H)=4$ and $Cyc(H)=1$. Thus
$$H\cong \underbrace{\mathbb{Z}_2 \oplus \cdots \oplus
\mathbb{Z}_{2}}_k \oplus \underbrace{\mathbb{Z}_{3} \oplus \cdots
\oplus \mathbb{Z}_{3}}_\ell.$$ Since $\omega(\mathcal{C}_H)=4$ and
$Cyc(H)=1$, it follows from Lemmas \ref{ele} and \ref{cl-cl} that
$k=0$ and $\ell=2$, that is,
 $H\cong \mathbb{Z}_3 \oplus \mathbb{Z}_3$, as required.
\end{proof}
\begin{lem}\label{27}
There is no  non-cyclic group $G$ of order $27$ with
$\omega(\mathcal{C}_G)=4$.
\end{lem}
\begin{proof}
Suppose, for a contradiction, that $G$ is a group of order 27
with $\omega(\mathcal{C}_G)=4$.  It follows from Lemmas \ref{2.6}
and \ref{ele} that $G$ is a non-abelian group of exponent $9$.
Therefore
\begin{eqnarray*} G\cong \langle
c, d\mid c^9= d^3=1, d^{-1}cd=c^4\rangle.
\end{eqnarray*}
Now the set $\{c, d, cd, c^{-1}d, cd^{-1}\}$ is a clique.  This
completes the proof.
\end{proof}
\begin{thm} \label{2.7} Let $G$ be a non locally cyclic group. Then
$\omega(\mathcal{C}_G)=4$ if and only if $G/Cyc(G)\cong
\mathbb{Z}_3\oplus \mathbb{Z}_3$ or $S_3$.
\end{thm}
\begin{proof} Suppose that $\omega(\mathcal{C}_G)=4$. Then by Lemma \ref{w-cyc} we
may assume that $Cyc(G)=1$. Also it follows from Lemma \ref{2.3}
that $G$ is a $\{2, 3\}$-group. By Lemma \ref{2.6}, we may assume
that $G$ is non-abelian. Now since $4=w(\mathcal{C}_G)\geq
w(\mathcal{A}_G)$, we have
$\omega(\mathcal{A}_G)=3$ or 4.\\
If $\omega(\mathcal{A}_G)=3$, then  $\frac{G}{Z(G)}\cong
\mathbb{Z}_2\oplus \mathbb{Z}_2$. In this case by an argument
similar to the proof of Theorem \ref{2.5} we obtain a
contradiction. So $\omega(\mathcal{A}_G)=4$, and by \cite[Lemma
2.4-(4)]{AJM} and \cite[Theorem 5]{BS}, we have  $G/Z(G)\cong
\mathbb{Z}_3\oplus \mathbb{Z}_3$ or $S_3$. By Lemma \ref{27},
Sylow 3-subgroups of $G$ are of order 3 or 9. By Lemma
\ref{G>G/N},   $Z(G)$ is cyclic. Therefore  $Z(G)\cong
\mathbb{Z}_2$ or $\mathbb{Z}_3$.\\
1) \; Let $G/Z(G)\cong \mathbb{Z}_3\oplus \mathbb{Z}_3$. Then
$\gcd(|G/Z(G)|,|Z(G)|)=1$, as  Sylow subgroups of $G$ are of order
at most 9. Thus   $Z(G)=Cyc(G)=1$ and so
$G$ is abelian, a contradiction. \\
2) \; Let $\frac{G}{Z(G)}\cong S_3$.  \\
(I) \; If $Z(G)\cong \mathbb{Z}_3$, then $|G|=18$, and since $G$
is not abelian $G\cong D_{18}$, $\langle c,d,e\mid c^3=d^3=e^2=1,
cd=dc, c^e=c^{-1}, d^e=d^{-1}\rangle$, or $\mathbb{Z}_3\times
S_3$. The first two have trivial centers, so $G\cong
\mathbb{Z}_3\times S_3$. But if $\mathbb{Z}_3=\langle z\rangle$
and $S_3=\langle a,b\mid a^3=b^2=1, a^b=a^{-1}\rangle$, then $\{a,
b, ab, a^2b, za\}$ is a clique, which is a contradiction.\\
 (II) \; Let $Z(G)\cong \mathbb{Z}_2$. Then $|G|=12$, and since $G$
 is non-abelian $G\cong A_4$, $D_{12}$ or $\langle a, b\mid a^6=1,
b^2=a^3, bab^{-1}=a^{-1}\rangle$. But $G\ncong A_4$, since
$Z(A_4)=1$. If $G\cong D_{12}= \langle c, d\mid c^6=1=d^2,
dcd^{-1}=c^{-1}\rangle$, then $\{c, d, cd, c^2d, c^3d\}$ is a
clique, a contradiction. If $G\cong \langle a, b\mid a^6=1,
b^2=a^3, bab^{-1}=a^{-1}\rangle$, then $\{a, b, ab, a^{-1}b, a^2b\}$
is a clique, a contradiction.\\
This completes the proof.
\end{proof}
\section{Planar and Hamiltonian non-cyclic graphs}
We were unable to decide  whether the non-cyclic graph of a
finite group is Hamiltonian or not. On the other hand, since the
non-cyclic graph of a finite group is so rich in edges, it is
hard to believe it is not Hamiltonian.\\
The following result  reduces the verification of being
Hamiltonian of the non-cyclic graph of a finite group $G$ to that
of the graph $\mathcal{C}_{G/Cyc(G)}$.
\begin{lem}
Let $G$ be a finite non-cyclic group such that
$\mathcal{C}_{G/Cyc(G)}$  is Hamiltonian. Then $\mathcal{C}_G$ is
also Hamiltonian. \end{lem}
\begin{proof}
By hypothesis there exists a cycle
$$\bar{a_1}-\bar{a_2}-\cdots-\bar{a_n}-\bar{a_1} \eqno{(\bigstar)}$$ in $\mathcal{C}_{\frac{G}{Cyc(G)}}$ such that
$$\frac{G}{Cyc(G)}\backslash \frac{Cyc(G)}{Cyc(G)}=\{\bar{a_1},\bar{a_2},\dots,\bar{a_n}\}.$$ Let
$Cyc(G)=\{c_1,\dots,c_k\}$. By $(\bigstar)$, $a_{i}c_j$ is
adjacent to $a_{i+1}c_\ell$ in $\mathcal{C}_G$ for all
$i\in\{1,\dots,n\}$ and all $j,\ell\in\{1,\dots,k\}$, where
indices of $a$'s are computed modulo $n$. Thus
$$a_1c_1-\cdots-a_nc_1-a_1c_2-\cdots-a_nc_2-\cdots-a_1c_k-\cdots-a_nc_k-a_1c_1$$
is a Hamilton cycle in $\mathcal{C}_G$. This completes the proof.
\end{proof}
In the following result we give a large family of finite groups
with Hamiltonian non-cyclic graphs.
\begin{prop}\label{ham}
Let $G$ be a finite non-cyclic group such that $$|G|+|Cyc(G)|>
2|Cyc_G(x)|$$ for all $x\in Z(G)\backslash Cyc(G)$. Then
$\mathcal{C}_G$ is Hamiltonian. In particular, $\mathcal{C}_G$ is
Hamiltonian whenever $Z(G)=Cyc(G)$.
\end{prop}
\begin{proof}
First note that the degree  of any vertex $x$ in  $\mathcal{C}_G$
 is equal to $|G\backslash Cyc_G(x)|$.
We now prove that $|G\backslash Cyc_G(x)|>
\frac{|G|-|Cyc(G)|}{2}$ for all $x\in G\backslash Cyc(G)$.
Suppose, for a contradiction, that $|G\backslash Cyc_G(x)|\leq
\frac{|G|-|Cyc(G)|}{2}$ for some $x\in G\backslash Cyc(G)$. It
follows that $$ 2|Cyc_G(x)|\geq |G| +|Cyc(G)|. \eqno{(*)}$$ It
now follows from $(*)$ that $2|C_G(x)|\geq |G| +|Cyc(G)|$. Since
$|C_G(x)|$ divides $|G|$, we have $|G|=|C_G(x)|$ and so $x\in
Z(G)$. Now $(*)$ contradicts our hypothesis,  as $x$ belongs to
$Z(G)\backslash Cyc(G)$. Therefore $d(x)>
\left(|G|-|Cyc(G)|\right)/2$ for all vertices $x$ of
$\mathcal{C}_G$. Hence by Dirac's theorem \cite[p.54]{bon},
$\mathcal{C}_G$ is Hamiltonian.
\end{proof}
The inequality stated in Proposition \ref{ham} does not hold in
general. For example if $G=C_6 \times S_3$, and $x\in C_6$ is an
element of order $3$, then it is easy to see that $|Cyc_G(x)|=24$
and as $Cyc(G)=1$, we see that the inequality does not hold.
\begin{prop}
Let $G$ be a non locally cyclic group. Then $\mathcal{C}_G$ is
planar if and only if $G$ is isomorphic to $\mathbb{Z}_2\oplus
\mathbb{Z}_2$, $S_3$ or $Q_8$.
\end{prop}
\begin{proof}
It is easy to see that the non-cyclic graphs of the groups stated
in the lemma  are all planar. Now suppose that $\mathcal{C}_G$ is
planar. Since the complete graph of order $5$ is not planar, we
have $\omega(\Gamma_G)<5$. Thus $G/Cyc(G)$ is isomorphic to
$\mathbb{Z}_p\oplus \mathbb{Z}_p$ or $S_3$, where $p\in\{2,3\}$,
by Theorems \ref{2.5} and \ref{2.7}. Now we prove that
$|Cyc(G)|\leq 2$. Suppose, for a contradiction, that $|Cyc(G)|>2$
and consider a finite subset $C$ of $Cyc(G)$ with $|C|=3$. Let
$x$ and $y$ be two adjacent vertices in $\mathcal{C}_G$. Put
$T=Cx\cup Cy$. Now the induced subgraph $\mathcal{C}_0$ of
$\mathcal{C}_G$ by $T$ is a planar graph. On the other hand,
$\mathcal{C}_0$ is isomorphic to the bipartite graph $K_{3,3}$, a
contradiction, since $K_{3,3}$ is not planar. If $G/Cyc(G)\cong
\mathbb{Z}_2\oplus \mathbb{Z}_2$, then $G\cong \mathbb{Z}_2\oplus
\mathbb{Z}_2$ or $Q_8$. If $G/Cyc(G)\cong \mathbb{Z}_3 \oplus
\mathbb{Z}_3$, then there are  two adjacent vertices $x$ and $y$
such that the orders of $xCyc(G)$ and $yCyc(G)$ are both $3$. Let
$I=\{x,x^{-1},y\}$ and $J=\{xy,x^{-1}y,(xy)^{-1}\}$. In the non-cyclic graph  $\mathcal{C}_G$ every vertex of $I$ is adjacent to
every vertex of $J$. Therefore $\mathcal{C}_G$ contains a copy of
$K_{3,3}$, a contradiction.

If $G/Cyc(G)\cong S_3$, then there are vertices $a$ and $b$ such
that $a\not=a^{-1}$ and  both $a$ and $a^{-1}$ are adjacent to
each vertex in $\{b,ab,a^2b\}$. Now suppose, for a contradiction,
that $Cyc(G)$ contains a non-trivial element $c$. Then
$\{a,a^{-1},ac\}$ and $\{b,ab,a^2b\}$ are the parts of a subgraph
of $\mathcal{C}_G$ isomorphic to $K_{3,3}$, a contradiction.
Therefore, in this case, $Cyc(G)=1$ and so $G\cong S_3$. This
completes the proof.
\end{proof}
\section{\bf A solvability criterion and  new characterizations for the symmetric and alternating groups of degree $5$}
We need the following result in the proof of Theorem \ref{sol}
below.
\begin{prop}\label{l2(p)}{\rm (Proposition 2.6 of  \cite{AM2})} Let $p$ be a prime number, $n$ a positive
integer and $r$ and $q$ be two odd prime numbers dividing
respectively $p^n + 1$ and $p^n - 1$. Then the number of Sylow
$r$-subgroups {\rm(}respectively, $q$-subgroups{\rm)} of $\mathrm{L}_2(p^n)$
is $p^n(p^n-1)/2$ {\rm(}respectively, $p^n(p^n+1)/2${\rm)}. Also any  two distinct Sylow $r$-subgroups or
$q$-subgroups have  trivial intersection.
\end{prop}
\begin{proof}
The proof follows from  Theorems 8.3 and 8.4 in chapter II of
\cite{Hup}. For a complete proof see the proof of Proposition 2.6
of \cite{AM2}.
\end{proof}
\begin{lem}\label{26} Let $G$ be  one of the following groups:\\
 $\mathrm{L}_2(2^p)$, $p=4$ or a prime; $\mathrm{L}_2(3^p)$, $\mathrm{L}_2(5^p)$, $p$ a
prime; $\mathrm{L}_2(p)$, $p$ a prime $\geq 7$; $\mathrm{L}_3(3)$,
$\mathrm{L}_3(5)$; $\mathrm{PSU}(3,4)$ {\rm(}the projective special
unitary group of degree $3$ over the finite field of order $4^2${\rm)}
or $\mathrm{Sz}(2^p)$, $p$ an odd prime.\\ Then
$\omega(\mathcal{C}_G)> 31$.
\end{lem}
\begin{proof} For every group listed above  we find a set $S$ of prime
numbers $p$ for which   Lemma \ref{2.2} is applicable. \\
 For every prime number $p$ and every
integer $n>0$, we have
 that the number of Sylow
$p$-subgroups of $\mathrm{L}_2(p^n)$  which are elementary abelian,
is $p^n+1$ and  any two distinct Sylow
$p$-subgroups have trivial intersection (see chapter II Theorem 8.2 (b),(c) of
\cite{Hup}). It follows that
$\nu_p(\mathrm{L}_2(p^n))=(p^n+1)(p^n-1)/(p-1)$. Thus among the
projective special linear groups, we only need to investigate the
following groups: $\text{L}_3(3)$, $\text{L}_3(5)$,
$\text{L}_2(p)$ for $p\in \{7,11,13,17,19,23,29\}$. Now if in
Proposition \ref{l2(p)}, we take $q=3$ for $\text{L}_2(13)$ and
$\text{L}_2(19)$; and $r=3$ for $\text{L}_2(11)$,
$\text{L}_2(17)$, $\text{L}_2(23)$ and $\text{L}_2(29)$; Then by
Lemma \ref{2.2} we are done in these cases. Therefore we must
consider the groups $\text{L}_2(7)$, $\text{L}_3(3)$,
$\text{L}_3(5)$, $\text{PSU}(3,4)$ or $\text{Sz}(2^p)$, $p$ an odd prime.\\
If $G=\text{L}_2(7)$, then $|G|=2^3 \times 3 \times 7$ and $G$
has no element of order $3\times 7$. Now it follows from Lemma
\ref{2.2}, that $\nu_3(G)+ \nu_7(G)\leq w(\mathcal{C}_G)$. Now by
Proposition \ref{l2(p)}, we have $\nu_3(G)=28$ and $\nu_7(G)=8$
and so $\omega(\mathcal{C}_G)\geq 36$. If $G=\text{L}_3(3)$, then
$|G|=2^4\times 3^3\times 13$ and $G$ has no element of order
$3\times 13$. Thus $\nu_{13}(G)=1+13k$ and $\nu_3(G)=1+3\ell$, for
some $k>0$ and $\ell>0$. Since $14$ does not divide $|G|$ and no
non-abelian simple group contains a subgroup of index less than 5,
$\nu_{13}(G)\geq 27$ and $\nu_3(G)\geq 7$. Now it follows from
Lemma \ref{2.2} that $\omega(\mathcal{C}_G)\geq 34$. If
$G=L_3(5)$, then $|G|=2^5\times 3\times 5^3 \times 31$. Thus
$\nu_{31}(G)=1+31k$, for some $k>0$ and so $\nu_{31}(G)>31$. If
$G=\text{PSU}(3,4)$, then $|G|=2^6\times 3\times 5^2\times 13$ (see
Theorem 10.12(d) of chapter II in \cite{Hup} and note that
$\text{PSU}(3,4)$ is the projective special unitary group of
degree 3 over the finite field of order $4^2$). Therefore
$\nu_{13}(G)= 1+13k$ for some $k>0$ and since $14$ does not
divide $|L|$, $\nu_{13}(L)>26$. If $G=\mathrm{Sz}(2^p)$ ($p$ an odd
prime), then it follows from  Theorem 3.10 (and its proof) of
chapter XI in \cite{Hup} that $\nu_2(G)\geq 2^{2p}+1\geq 65$.
This completes the proof.
\end{proof}
\begin{thm}\label{sol}
Let $G$ be a  non locally cyclic group.
\begin{enumerate}
\item If   $\omega(\mathcal{C}_G)=31$, then $G$ is simple if and only if
$G\cong A_5$.
\item If  $\omega(\mathcal{C}_G)\leq 30$, then $G$ is solvable.
\end{enumerate}
\end{thm}
\begin{proof}
1) \; It follows from \cite[Theorem 4.2]{AM} that $G/Cyc(G)$ is
finite. Thus, if $G$ is simple, it is finite. Now  suppose, for a
contradiction, that there exists a non-cyclic finite simple group
$K$ with $\omega(\mathcal{C}_K)=31$ which is not isomorphic to
$A_5$. Let $T$ be such a group of least order. Thus every proper
non-abelian simple section of $K$ is isomorphic to $A_5$.
Therefore by Proposition 3 of \cite{BR}, $T$ is isomorphic to one
of the groups in the statement of Lemma \ref{26}, which is
impossible. This implies that $G\cong A_5$.\\
Now we prove that $\omega(\mathcal{C}_{A_5})=31$. Note that the
order of  an element of $A_5$ is $2$, $3$ or $5$; $A_5$ has five
Sylow $2$-subgroups, ten  Sylow $3$-subgroups and six Sylow
$5$-subgroups; and any two distinct Sylow subgroups has trivial
intersection. Now consider the set of all non-trivial 2-elements
of $A_5$ and select one group generator from each  Sylow
$p$-subgroup for $p\in\{3,5\}$. Then the union $C$ of these sets
is of size $31$ and every two distinct element of $C$ generate a
non-cyclic subgroup. On the other hand, since $A_5$ is the union
of its Sylow subgroups, it is easy to show that every clique set
of $\mathcal{C}_{A_5}$ is of size at most $31$. This completes
the proof of (1). \\

\noindent 2) \; It follows from \cite[Theorem 4.2]{AM} that
$G/Cyc(G)$ is finite and by Lemma \ref{w-cyc}, we may assume that
$G$ is finite. Let $K$ be a counter-example of the least order.
Thus every proper subgroup of $K$ is solvable and $G$ is a
non-abelian simple group. That is to say, $K$ is a minimal simple
group.  Thus according to Thompson's classification  of the
minimal simple groups in \cite{Thomp}, $K$ is
isomorphic to one of the following:\\
 $\text{L}_2(p)$ for some prime $p\geq 5$, $\text{L}_2(2^p)$ or  $\text{L}_2(3^p)$ for some prime $p\geq
 3$,  $\text{Sz}(2^p)$ for some prime $p\geq 3$, or
 $\text{L}_3(3)$.\\ By Lemma \ref{l2(p)} and part (1) we have that
 $\omega(\mathcal{C}_K)\geq 31$ and this contradicts the hypothesis.
 Hence $G$ is solvable.
\end{proof}
\begin{lem}\label{S5=A5} Let $G$ be either $A_5$ or $S_5$. Then
$\omega(\mathcal{C}_{G})=31$. Moreover, every non-trivial element
of $G$ is contained in a maximum clique of $\mathcal{C}_G$.
\end{lem}
\begin{proof}
Clearly $\omega(\mathcal{C}_{S_5})\geq w(\mathcal{C}_{A_5})$. Let
$C$ be the clique found in Theorem \ref{sol} for
$\mathcal{C}_{A_5}$. It is easy to see (e.g., by {\sf GAP}
\cite{GAP}) that
$$G=\bigcup_{x\in C} Cyc_G(x) \;\;\text{and}\;\; Cyc_G(x) \;\text{is a cyclic subgroup for
all}\;\; x\in C$$  $$ \text{and} \;\; Cyc_G(x)=Cyc_G(a) \;\;
\text{for all non-trivial} \;  a\in Cyc_G(x). \eqno{(*)}$$ Thus
$G$ is the union of 31 cyclic subgroups and so
$\omega(\mathcal{C}_{S_5})\leq 31$. It follows that
$\omega(\mathcal{C}_{G})=31$. The second part follows easily from
$(*)$. This completes the proof.
\end{proof}
\begin{thm}\label{non-sol-31}
Let $G$ be a  non-solvable group. Then $\omega(\mathcal{C}_G)=31$
if and only if $G/Cyc(G)\cong A_5$ or $S_5$.
\end{thm}
\begin{proof}
If $G/Cyc(G)\cong A_5$ or $S_5$, then it follows from Lemmas \ref{w-cyc} and \ref{S5=A5} that $\omega(\mathcal{C}_G)=31$. \\
Suppose that $\omega(\mathcal{C}_G)=31$. By Theorem 4.2 of
\cite{AM}, $G/Cyc(G)$ is finite. Thus we may assume without loss
of generality that $G$ is finite and $Cyc(G)=1$ and so we have to
prove
$G\cong A_5$ or $S_5$.\\
 Let $S$ be the largest normal solvable subgroup of $G$ (here, for the existence of $S$  we use the finiteness of $G$).
Then $\overline{G}=G/S$ has no non-trivial abelian normal
subgroup. Let $R$ be the product of all minimal normal non-abelian
subgroups of $\overline{G}$. It follows from Lemma 2.1 of
\cite{AM2} and Theorem \ref{sol} that $R\cong A_5$. Since
$C_{\overline{G}}(R)=1$, we have that $\overline{G}$ is
isomorphic to a subgroup of $S_5$. It follows that
$\overline{G}\cong A_5$ or $S_5$. Now it follows from  Lemmas
\ref{G>G/N} and  \ref{S5=A5}, that $S\leq Cyc(G)=1$. This
completes the proof.
\end{proof}
We remark here that there are solvable groups $G$ for which
$\omega(\mathcal{C}_G)=31$; for example, by Lemma \ref{ele}, we
may take $G$ to be either the elementary abelian $2$-group of
rank 5
or the elementary abelian $5$-group of rank 3. \\

We end the paper with the answer of   following question posed in
\cite{AM}.
\begin{qu}\label{qu}{\rm (Question 2.4 of \cite{AM})} Let $G$ be a torsion free group such that $Cyc(G)$ is
non-trivial. Is it true that G is locally cyclic?
\end{qu}
In \cite[Theorem 31.4]{O} Ol'shanskii has constructed   a
non-abelian torsion-free group $G$   all of whose proper
subgroups are cyclic and it is central extension of an infinite
cyclic group $Z$ by an infinite group of bounded exponent. Since
the group $G$ is 2-generated,  it is not  locally cyclic. Also
$Z\leq Cyc(G)$, for if $z\in Z$ and $a\in G$, then $\langle
z,a\rangle$ is abelian, and as $G$ is not abelian, $\langle
z,a\rangle\not=G$. Thus  $\langle z,a\rangle$ is cyclic. Hence
the answer of Question 2.4 of \cite{AM} is negative.\\

\noindent{\bf Acknowledgment.} The authors are indebted to the
referee for his/her careful reading and invaluable comments.

\end{document}